\documentclass[a4paper,11pt]{article}
\usepackage[utf8]{inputenc}
\usepackage{amsmath}
\usepackage{amsthm}
\usepackage{amsfonts}
\usepackage{amssymb}
\usepackage{mathrsfs}
\usepackage{tikz}
\numberwithin{equation}{section}

\newtheorem{theorem}{Theorem}
\newtheorem{corollary}[theorem]{Corollary}
\newtheorem{proposition}[theorem]{Proposition}
\newtheorem{lemma}[theorem]{Lemma}

\theoremstyle{definition}\newtheorem{Example}[theorem]{Example}

\title{Non-symplectic actions on complex projective spaces}
\author{Marek Kaluba and Wojciech Politarczyk}

\begin{document}

\maketitle
\begin{abstract}
\noindent
We construct smooth actions of arbitrary compact Lie groups on complex projective spaces, such that the corresponding transformations arising from the group action do not preserve any symplectic structure on the complex projective space.
\end{abstract}

\section{Main Theorems}

Let $G$ be a compact Lie group and let $M$ be a symplectic manifold. Then one may ask whether or not, every smooth action of $G$ on $M$\! is symplectic, i.e., there exists a symplectic form $\omega$ on $M$ such that ${\theta_g}^{\ast} \omega = \omega$ for all $g \in G$, where $\theta_g \colon M \to M$ is given by $\theta_g(x) = gx$ for all $x \in M$.

The first examples of closed symplectic manifolds $M$ with non-symplectic smooth $G$-actions go back to Hajduk, Pawałowski, and Tralle
\cite{Hajduk-Pawalowski-Tralle}, where the $G$-fixed point set and the $H$\nobreakdash-isotropy point set are not symplectic manifolds
for $G=S^1$ and $H=\mathbb{Z}/_{pq}$, where $(p,q)=1$. According to \cite[Lemma~2.1]{Hajduk-Pawalowski-Tralle}, actions with such $G$-fixed points sets are not
symplectic.

Recall that Audin \cite[pp.~104--111]{Audin} has described oriented compact smooth $4$\nobreakdash-manifolds $M$ which admit smooth actions of the circle $S^1$
but do not admit any symplectic forms preserved by the actions of $S^1$. However, it is not clear whether those $4$-manifolds $M$ are symplectic 
at all.

In this paper, for any compact Lie group $G$, we construct non-symplectic smooth actions 
of $G$ on complex projective spaces using the conclusions of Theorems \ref{thm:main1} and 
\ref{thm:main2} below.

\begin{theorem}\label{thm:main1}
Let $G$ be a compact Lie group and let $F$ be a compact $\mathbb{Z}$-acyclic smooth $(2d+1)$\nobreakdash-manifold for $d \geq 1$. Then there exists a smooth action of\/ $G$ on a complex projective space such that the fixed point set is diffeomorphic to the disjoint union of\/ $\mathbb{C}P^d \# \partial F$ and a number of complex projective spaces.
\end{theorem}

For any compact Lie group $G$, Theorem~\ref{thm:main1} allows to obtain
non-symplectic smooth actions of $G$ on complex projective spaces.
In fact, consider a smooth $\mathbb{Z}$-homology $4$\nobreakdash-sphere $N$, a closed
smooth $4$\nobreakdash-manifold with $\mathbb{Z}$-homology of the $4$\nobreakdash-sphere.
According to Kervaire \cite[Theorem~3, p.~70]{Kervaire}, $N$ is diffeomorphic to the boundary $\partial F$ of a contractible smooth $5$\nobreakdash-manifold $F$.
Therefore, by Theorem~\ref{thm:main1}, there exists a smooth action of $G$ on  a complex projective space such that the fixed point set contains a connected component diffeomorphic to $\mathbb{C}P^2 \# N$, which is an almost complex manifold due to \cite[Theorem~1.4.15]{Gompf-Stipsicz}. However, using the celebrated result of Taubes \cite{Taubes}, we prove that $N$ can be chosen so that $\mathbb{C}P^2 \# N$ is not a symplectic manifold and hence, the action is not symplectic (see Proposition~\ref{prop:not_symplectic} and Example~\ref{expl:hmlgy sphereCPn}).

Let $G$ be a compact Lie group such that the identity connected component $G_0$ of $G$
is non-abelian or $G/G_0$ is not of prime power order, i.e., $G$ is not isomorphic to a torus, 
finite $p$-group or its extension by a torus. Then, by the work of Oliver \cite{Oliver:1} and \cite{Oliver:2}, 
there exists an integer $n_G \geq 0$ which is the only obstruction for a finite CW complex $F$ to occur 
as the fixed point set in a finite contractible $G$-CW complex $X$. In fact, there exists such an $X$ 
with $X^G = F$ if and only if $\chi(F)\equiv 1 \pmod {n_G}$. We refer to $n_G$ as to the \emph{Oliver number} 
of $G$. Recall that $n_G$ has been computed by Oliver \cite{Oliver:3} when $G$ is finite. In general, $n_G = n_{G/G_0}$ 
when $G_0$ is abelian (i.e., $G_0$ is a torus or the trivial group), and $n_G = 1$ when $G_0$ is non-abelian.

\begin{theorem}\label{thm:main2}
Let $G$ be a compact Lie group such that $G_0$ is non-abelian or $G/G_0$ is not of prime power order. Let $F$ be a compact stably parallelizable smooth $(2d+1)$\nobreakdash-manifold for $d \geq 1$. Assume that $\chi(F) \equiv 1 \pmod {n_G}$. 
Then there exists a smooth action of $G$ on a complex projective space such that the fixed point set is diffeomorphic 
to the disjoint union of $\mathbb{C}P^d \# \partial F$ and a number of complex projective spaces.
\end{theorem}

Clearly, in Theorem \ref{thm:main2}, $\mathbb{C}P^d \# \partial F\cong(\mathbb{C}P^d \# N_0)\sqcup N_1\sqcup\ldots\sqcup N_k$ 
for the connected components $N_0,N_1,\ldots, N_k$ of $\partial F$. If we choose $F$ in such a way that some of the $N_i$'s are not symplectic manifolds, then Theorem~\ref{thm:main2} yields examples of non-symplectic smooth actions on complex projective spaces (see Example \ref{expl:euler characteristic} and the subsequent discussion).

For any group $G$ in Theorem \ref{thm:main2}, the conclusion of Theorem \ref{thm:main1} follows 
from Theorem \ref{thm:main2} (as any $\mathbb{Z}$\nobreakdash-acyclic smooth manifold $F$ is stably parallelizable and the Euler characteristic $\chi(F)=1$).

In both theorems, we argue as follows. First, we start with a construction of a smooth action of $G$ on a disk $D$ with a prescribed fixed point set $F$. Then, we restrict the action to the boundary $S=\partial D$ and, finally, we form the connected sum of $S$ and the complex projective space with appropriate linear action of $G$. 

Using the arguments above, we can obtain smooth actions of $G$ on complex projective spaces for more general manifolds $F$ than in Theorems \ref{thm:main1} and \ref{thm:main2}. 

For example,  we could have taken the fixed point sets of actions on disks described in the articles of Oliver \cite[Theorems 0.1 and 0.2]{Oliver:4} or Pawa{\l}owski \cite[Theorem~A]{Pawalowski:1}. However, for the sake of simplicity, we decided to restrict our attention to the case where $F$ is stably parallelizable. This allows us to deal only with  product $G$\nobreakdash-vector bundles when applying the equivariant thickening procedure (cf. Theorem \ref{thm:thickening}).\\[0.1in]
The remaining part of the material is organized as follows.
In Section~\ref{sec:2-examples}, we present two examples justifying that by Theorems \ref{thm:main1} and \ref{thm:main2}, non-symplectic smooth manifolds can be realized as the fixed point sets of smooth $G$-actions on complex projective spaces. 

In Section~\ref{sec:grp-actions-on-D-and-S}, for any compact Lie group $G$, we construct smooth actions of $G$ on disks $D$ with a given stably parallelizable manifold $F$ occurring as the fixed point set. In Section~\ref{sec:grp-actions-on-CPn}, by restricting the actions of $G$ on $D$ to its boundary $S$ and taking the equivariant connected sum of the appropriate complex projective space and $S$, we obtain smooth actions of $G$ which show that Theorems~\ref{thm:main1} and \ref{thm:main2} both hold.

We refer the reader to the textbooks by Bredon \cite{Bredon}, tom Dieck \cite{tomDieck}, or Kawakubo \cite{Kawakubo}
for the background material on transformation groups that we use in this paper.\\[0.2in]
\noindent
{\sl Acknowledgements.} We express our thanks to Krzysztof Pawa{\l}owski for his information on the existence of
non-symplectic smooth actions on symplectic manifolds as in \cite{Hajduk-Pawalowski-Tralle}, and his help in obtaining
the results of this paper. Also, we would like to express our gratitude to the referee for his comments which allowed us
to improve the presentation of the material.

\section{Examples of non-symplectic manifolds}\label{sec:2-examples}
For any compact Lie group $G$, the existence of non-symplectic smooth actions of $G$ on complex projective spaces follows from Theorem~\ref{thm:main1}, Proposition~\ref{prop:not_symplectic}, and Example~\ref{expl:hmlgy sphereCPn}.
Hereafter, for a closed $4$\nobreakdash-manifold $M$, $b_2^+(M)$ is the number of positive entries in a diagonalization of the intersection form of $M$ over $\mathbb{Q}$.
By the work of Taubes \cite{Taubes}, for two closed smooth $4$\nobreakdash-manifolds $M$ and $M'$ with $b_2^+ > 0$, the connected sum $M \# M'$ never admits a symplectic structure (cf. \cite[Proposition~10.1.14]{Gompf-Stipsicz}). We shall make use of this result in the proof of the following proposition (cf. \cite[Proposition 1]{Kotschick-Morgan-Taubes}).

\begin{proposition}\label{prop:not_symplectic}
Let $M$ and $N$ be two closed oriented smooth $4$-manifolds such that $b_2^+(M)>0$ and $\pi_1(N)$ has a subgroup of finite index $r>1$. 
Then the connected sum $M\#N$ is not a symplectic manifold.
\end{proposition}

\begin{proof}
As $\pi_1(N)$ has a subgroup of  finite index $r>1$, there exists a $r$-sheeted covering $\widetilde{N}\to N$. The $r$\nobreakdash-sheeted covering $E\to M\#N$ has the form
\[E=rM \# \widetilde{N}=M\# M'\]for $M'=(r-1)M\# \widetilde{N}$. Clearly, $b_2^+(M')\geq b_2^+(M) > 0$ and thus, it follows from \cite[Proposition~10.1.14]{Gompf-Stipsicz} that $E$ is not a symplectic manifold. Therefore, $M\# N$ is not a symplectic manifold either, as otherwise, the space $E$ would admit a symplectic form induced via the covering $E \to M\# N$.
\end{proof}

\begin{Example}\label{expl:hmlgy sphereCPn}
Set $M=\mathbb{C}P^2$ and let $N$ be the smooth homology $4$-sphere constructed by Sato \cite[Example~3.3]{Sato}, where $\pi_1(N)$ is isomorphic to the group $SL(2,5)$ of order $120$. As $b_2^+(\mathbb{C}P^2)=1$, it follows from Proposition \ref{prop:not_symplectic} that $M\#N$ is not a symplectic manifold.
\end{Example}

The following example justifies that for the goal of providing examples of non-symplectic smooth actions 
by using Theorem \ref{thm:main2}, the Euler characteristic condition that $\chi(F) \equiv 1 \pmod {n_G}$ 
is not restrictive at all.

\begin{Example}\label{expl:euler characteristic}
For an integer $d\geq 1$, we find a compact parallelizable smooth $(2d+1)$-manifold $F$ such that $\chi(F)=1$ and $\partial F$
contains as many connected components of the form $S^{2d}$ or $S^p\times S^q$ for $p+q=2d$, as we wish.

For an integer $g \geq 0$, consider the compact parallelizable $3$-manifold $M_g$ of the homotopy type of the wedge of $g$ circles,
whose boundary $\partial M_g$ is the orientable surface of genus $g$.

For an integer $h \geq 0$, let $M_{g,h}$ denote the manifold obtained from $M_g$ by removing $h$ disjoint open balls from
the interior of $M_g$. Then\[\chi(M_{g,h} \times D^{2d-2}) = 1 - g + h.\]

For two non-negative integers $k$ and $\ell$, let $F$ denote the disjoint union of $k$ copies of $D^{2d+1}$,
$\ell$ copies of $D^{p+1}\times S^q$, and just one copy of $M_{g,h} \times D^{2d-2}$.
Then $F$ is a compact parallelizable smooth $(2d+1)$-manifold such that\[\chi(F) = k + \ell \chi(S^q) + 1 - g + h.\]By choosing
$g$ and $h$ such that $g - h = k + \ell \chi(S^q)$, we see that $\chi(F)=1$ for the given integers $k, \ell, p, q, $ and $d$.
\end{Example}

For $d \geq 2$, $H^2(S^{2d}) = 0$ and therefore, $S^{2d}$ is not a symplectic manifold.
Even more, $S^{2d}$ does not admit an almost complex structure for $d \notin \{1,3\}$,
because $S^2$ and $S^6$ are the only spheres admitting such structures (cf. \cite{Borel-Serre}).

The products $S^2 \times S^4$, $S^2 \times S^6$, and $S^6 \times S^6$ cover the only cases where $S^p \times S^q$ admits
an almost complex structure for $p$ and $q$ both even (cf.~\cite{Datta-Subramanian}). Since $x^2=0$ for a generator $x$ of
$H^2(S^2\times S^4)$, as well as of $H^2(S^2\times S^6)$, the products $S^2\times S^4$ and $S^2\times S^6$
are not symplectic manifolds. Moreover, $S^6\times S^6$ is not a symplectic manifold, because $H^2(S^6 \times S^6)=0$.

By \cite[Example~4.30, p.~161]{Felix-Oprea-Tanre}, if $p$ and $q$ both are odd, $S^p \times S^q$ has the structure of
a complex (hence, almost complex) manifold, but for $p \geq 3$ or $q \geq 3$, it is not a symplectic manifold, because
$H^2(S^{p} \times S^{q}) = 0$.

It follows from \cite[Lemma~2.1 (resp., Lemma~2.3)]{Hajduk-Pawalowski-Tralle} that if we choose $F$ with connected components
of $\partial F$ as above, Theorem~2 yields smooth actions of $G$ which cannot preserve any symplectic (resp., almost complex)
structures on the complex projective spaces.

\section{Group actions on disks and spheres}\label{sec:grp-actions-on-D-and-S}

Let $G$ be a compact Lie group and let $V$ be a real $G$-module. For an integer $\ell \geq 1$, we set $\ell\,V = V \oplus \cdots \oplus V$, the direct sum of $\ell$ copies of $V$. 

For a $G$-space $X$, we denote by $\varepsilon^V_X$ the product $G$-vector bundle $X\times V $ over $X$. If $V =\mathbb{R}^d$ with the trivial action of $G$, we write $\varepsilon_X^d$ instead of $\varepsilon_X^V$. Also, we set $\mathscr{F}_{\text{iso}}(G;X)=\{G_x \colon x\in X\}$, where $G_x$ is the isotropy subgroup of $G$ at $x \in X$.

Any finite $G$-CW complex $X$ has a finite number of orbit types. Therefore, by \cite[Chapter~0, Theorem~5.2]{Bredon}, one can always choose a real (resp., complex) $G$-module $V$ such that $\mathscr{F}_\mathrm{iso}(G;X)\subset\mathscr{F}_\mathrm{iso}(G;V)$.

\begin{theorem}\label{thm:isotropy-subgroups}
Let $G$ be a compact Lie group and let $X$ be a $G$-CW complex with finite number of orbits types. Then there exists a real (resp., complex) $G$-module $V$ such that $\mathscr{F}_\mathrm{iso}(G;X)\subset\mathscr{F}_\mathrm{iso}(G;V)$ and $V^G = \{0\}$.
\end{theorem}

We recall that a smooth manifold $M$ is \emph{stably parallelizable} if the tangent bundle $\tau_M$ is stably trivial, i.e., $\tau_M \oplus \varepsilon^1_M \cong \varepsilon_M^{d+1}$ for $d = \dim M$. In particular, all connected components of $M$ must have the same dimension.

In order to construct smooth group actions on disks, we make use of the equivariant thickening procedure of Pawa{\l}owski \cite[Theorems~2.4 and 3.1]{Pawalowski:1} (cf.~\cite[Section~10]{Pawalowski:2}). A special case of the procedure looks as follows.

\begin{theorem}[The equivariant thickening for stably parallelizable manifolds] \label{thm:thickening}
Let $G$ be a compact Lie group. Let $X$ be a finite, contractible $G$\nobreakdash-CW complex such that the fixed point set $F$ is a stably parallelizable smooth $d$-manifold for an integer $d \geq 0$.
Let $V$ be a real $G$-module such that $V^G = \{0\}$ and $\mathscr{F}_\mathrm{iso}(G;X)\subset\mathscr{F}_\mathrm{iso}(G;V)$. Set $n = \ell\,\dim V$ for an integer $\ell \geq 1$. 

Then, if the integer $\ell$ is sufficiently large, there exists a smooth action of $G$ on the disk $D^{n+d}$ such that the following two conclusions hold.
\begin{enumerate}
\item[$(1)$] $(D^{n+d})^G = F$ and $\mathscr{F}_\mathrm{iso}(G;D^{n+d} \smallsetminus F)
             = \mathscr{F}_\mathrm{iso}(G;\ell \,V \smallsetminus \{0\})$.
\item[$(2)$] $\nu_{F \subset D^{n+d}} \cong \varepsilon_F^{\ell\,V}$ and therefore, as real $G$-modules, $T_x(D^{n+d}) \cong \ell\,V \oplus \mathbb{R}^d$ for every point $x \in F$, where $G$ acts trivially on $\mathbb{R}^d$.
\end{enumerate}
\end{theorem}

Now, we can see that a stably parallelizable compact smooth manifold $F$ occurs as the fixed point set of a smooth action of $G$ on a disk if and only if $F$ occurs as the fixed point set in a finite contractible $G$-CW complex.

\begin{corollary}\label{cor:manifold-realization}
Let $G$ be a compact Lie group. Let $F$ be a stably parallelizable, compact smooth manifold. Then the following two conditions are equivalent.
\begin{itemize}
\item[$(1)$] There exists a finite contractible $G$-CW complex $X$ with $X^G = F$.
\item[$(2)$] There exists a smooth action of $G$ on a disk $D$ with $D^G = F$.
\end{itemize}
\end{corollary}

\begin{proof}
By Theorems~\ref{thm:thickening} and \ref{thm:isotropy-subgroups}, (1) implies (2). Also, as any compact smooth $G$\nobreakdash-manifold has the structure of a finite $G$-CW complex, (2) implies~(1).
\end{proof}


\begin{proposition}\label{pro:G-CW-complex}
Let $G$ be a compact Lie group and let $F$ be a compact stably parallelizable
smooth manifold. Then under either of the five conditions:
\begin{itemize}
\item[$(1)$] $G$ is arbitrary and $F$ is contractible, and thus $\chi(F) = 1$,
\item[$(2)$] $G$ is connected and $F$ is $\mathbb{Z}$-acyclic, and thus $\chi(F) = 1$,
\item[$(3)$] $G/G_0$ is a $p$-group and $F$ is $\mathbb{Z}_p$-acyclic,
             and thus $\chi(F) = 1$,
\item[$(4)$] $G/G_0$ is not of prime power order and $\chi(F)\equiv 1 \pmod {n_G}$,
\item[$(5)$] $G_0$ is non-abelian and $\chi(F)$ is arbitrary,
\end{itemize}
there exists a finite contractible $G$-CW complex $X$ with $X^G = F$.
\end{proposition}

\begin{proof}
In the case (1), recall that any smooth contractible (more generally, $\mathbb{Z}$-acyclic) manifold is stably parallelizable. Since $F$ is contractible, so is the join $X = G \ast F$.
Clearly, $X$ with the join action of $G$ is a finite $G$-CW complex with $X^G = F$ (cf. \cite[Theorem~4.2, p. 282]{Pawalowski:1}).

In the case (2), $G$ is connected and $F$ is $\mathbb{Z}$-acyclic, and thus $X = G \ast F$ is simply connected and $\mathbb{Z}$-acyclic. Therefore, $X$ is contractible.

In the case (3), 
according to Jones~\cite{Jones}, there exists a finite contractible $G$-CW complex $X$
such that $X^{G/G_0} = F$.

In the case (4), $G/G_0$ is not of prime power order and $\chi(F)\equiv 1 \pmod {n_G}$.
Then, according to Oliver \cite{Oliver:1}, there exists a finite contractible $G/G_0$-CW
complex $X$ such that $X^{G/G_0} = F$.

In the cases (3) and (4), the epimorphism $G \to G/G_0$ allows to consider $X$ as a finite $G$-CW complex with $X^G = F$. By replacing $X$ by the product of $X$ and a disk with a linear effective action of $G$ fixing only the origin of the disk, one may assume that the action of $G$ on $X$ is effective.

In the case (5), $G_0$ is non-abelian and $\chi(F)$ is arbitrary. Then, according to Oliver \cite{Oliver:2}, there exists a finite contractible $G$-CW complex $Y$ such that $Y^G = \varnothing$.
Therefore, $X = Y \ast F$ is a finite contractible $G$-CW complex such that the fixed point set $X^G = F$.
\end{proof}


\section{Group actions on complex projective spaces}\label{sec:grp-actions-on-CPn}

Henceforth, $\mathbb{C}^n_{\varrho}$ (resp., $\mathbb{R}^n_{\varrho}$) denotes the complex
(resp., real) vector space $\mathbb{C}^n$ (resp., $\mathbb{R}^n$) with the linear action of
$G$ given by a unitary (resp., orthogonal) representation $\varrho \colon G \to U(n)$
(resp., $\varrho \colon G \to O(n)$).

\begin{lemma}\label{lem:module-realization}
Let $G$ be a compact Lie group and let $\varrho \colon G \to U(n)$ be a unitary representation
without trivial summand. Then, for any integer $d \geq 0$, there exists a smooth action
of $G$ on the complex projective space $\mathbb{C}P^{n+d}$ such that the fixed point set is
diffeomorphic to the disjoint union of the space $\mathbb{C}P^d$ and a~number of complex projective spaces. Moreover, for any $x \in \mathbb{C}P^d$,
\[T_x(\mathbb{C}P^{n+d})\cong \mathbb{C}^{n}_{\varrho} \oplus \mathbb{C}^{d}\]
as complex $G$-modules, where $G$ acts trivially on $\mathbb{C}^{d}$.
\end{lemma}

\begin{proof}
The action $\alpha \colon G \times \mathbb{C}^{n+d+1} \to \mathbb{C}^{n+d+1}$ given by
\[g \cdot (z_1,\ldots, z_{n+d+1}) = (\varrho(g)(z_1, \ldots, z_n), z_{n+1}, \ldots, z_{n+d+1})\]
induces a smooth action
$\overline{\alpha} \colon G \times \mathbb{C}P^{n+d} \to \mathbb{C}P^{n+d}$
via the diagram:
\begin{figure}[h]
\begin{center}
\begin{tikzpicture}[scale=0.9]
\draw (0,0) node(GxCP^n) {$G\times \mathbb{C}P^{n+d}$};
\draw (6,0) node(CP^n) {$\mathbb{C}P^{n+d}$};
\draw (0,2) node(GxC^n+1) {$G\times (\mathbb{C}^{n+d+1}\smallsetminus \{0,\ldots,0\}$)};
\draw (6,2) node(C^n+1) {$\mathbb{C}^{n+d+1}\smallsetminus\{0,\ldots,0\}$};
\path [->] (GxC^n+1) edge node [above] {$\alpha$} (C^n+1);
\path [->] (GxC^n+1) edge node [left] {${\operatorname{id}} \times \pi$} (GxCP^n);
\path [->] (GxCP^n)  edge node [above] {$\overline{\alpha}$} (CP^n);
\path [->] (C^n+1) edge node [right]{$\pi$} (CP^n);
\end{tikzpicture}
\end{center}
\end{figure}

\noindent
where $\pi$ is the quotient map.
It follows that $(\mathbb{C}P^{n+d})^G$ consists of
\[\mathbb{C}P^d = \{ [0,\ldots, 0, z_{n+1}, \ldots, z_{n+d+1}] \in \mathbb{C}P^{n+d}\}\]
and a number of complex projective spaces.

Consider the point $x=[0,\ldots, 0,1]$ in $(\mathbb{C}P^{n+d})^G$ together with its open neighbourhood
\[U = \{[z_1,\ldots, z_{n+d}, z_{n+d+1}] \colon z_{n+d+1} \neq 0\}\]
in $\mathbb{C}P^{n+d}$. The map $h:U\to\mathbb{C}^{n+d}$ defined by
\[h([z_1,\ldots, z_{n+d}, z_{n+d+1}]) =
\left(\frac{z_1 \texttt{}}{z_{n+d+1}},\ldots, \frac{z_{n+d}}{z_{n+d+1}}\right)\]
is a homeomorphism with the inverse $h^{-1}(z_1, \ldots,z_{n+d})=[z_1,\ldots,z_{n+d}, 1]$.
For any $g \in G$, let $\theta_g \colon \mathbb{C}^{n+d} \to \mathbb{C}^{n+d}$
and $\overline{\theta}_g \colon \mathbb{C}P^{n+d} \to \mathbb{C}P^{n+d}$ be given by
\[\theta_g(z_1, \ldots, z_{n+d}) = (\varrho(g)(z_1, \ldots, z_n), z_{n+1}, \ldots, z_{n+d}) \ \text{and} \]
\[\overline{\theta}_g([z_1, \ldots, z_{n+d+1}]) = [\varrho(g)(z_1, \ldots, z_n),
z_{n+1}, \ldots, z_{n+d+1}]\text{.}\]
Then $h \circ \overline{\theta}_g \circ h^{-1} = \theta_g
\colon \mathbb{C}^{n+d} \to \mathbb{C}^{n+d}$ and therefore,
\[D_{h(x)}(h \circ \overline{\theta}_g \circ h^{-1}) = D_{h(x)}(\theta_g) = \theta_g,\]
proving that as complex $G$-modules,
\[T_x(\mathbb{C}P^{n+d}) \cong \mathbb{C}^n_{\varrho} \oplus \mathbb{C}^d\]
where $G$ acts trivially on $\mathbb{C}^d$.
As $\mathbb{C}P^d$ occurs a connected component of the fixed point set $(\mathbb{C}P^{n+d})^G$,
the $G$-module $T_x(\mathbb{C}P^{n+d})$ does not depend on the choice of
$x \in \mathbb{C}P^d$, completing the proof.
\end{proof}

\begin{lemma}\label{lem:manifold-realization}
Let $G$ be a compact Lie group acting smoothly on a sphere $S^{2n+2d}$  for $n\geq 1$ and $d \geq 0$, with fixed point set $N$ of dimension $2d$. Suppose that for a point $y\in N$, as real $G$-modules, \[ T_y(S^{2n+2d}) \cong \mathbb{R}^{2n}_{\varrho} \oplus \mathbb{R}^{2d}\] where  $\varrho \colon G \to U(n) \subset O(2n)$ is a representation without trivial summand and $G$ acts trivially on $\mathbb{R}^{2d}$. Then there exists a smooth action of $G$ on the space $\mathbb{C}P^{n+d}$ such that the fixed point set is diffeomorphic to the disjoint union of\/ $\mathbb{C}P^d \# N$ and a number of complex projective spaces.

\end{lemma}

\begin{proof}
The representation $\varrho \colon G \to U(n)$ defines a linear action of $G$
on $\mathbb{C}^n$ without trivial summand. According to Lemma~\ref{lem:module-realization},
there exists a smooth action of $G$ on $\mathbb{C}P^{n+d}$ such that the fixed point set
is diffeomorphic to
\[\mathbb{C}P^d \sqcup \mathbb{C}P^{d_1} \sqcup \cdots \sqcup \mathbb{C}P^{d_k}\]
for some integers $d_1, \ldots, d_k \geq 0$. Moreover, as real $G$-modules, \[T_x(\mathbb{C}P^{n+d})\cong \mathbb{R}^{2n}_{\varrho} \oplus \mathbb{R}^{2d} \cong T_y(S^{2n+2d})\]for any $x \in \mathbb{C}P^d$.
So, the Slice Theorem allows us to identify some invariant disk neighbourhoods centred at
$x$ and $y$ in $\mathbb{C}P^{n+d}$ and $S^{2n+2d}$\!, respectively. Therefore, we can form
the connected sum
\[\mathbb{C}P^{n+d} \#_{x,y} S^{2n+2d}\cong \mathbb{C}P^{n+d}\!\,\]
upon which $G$ acts smoothly with fixed point set diffeomorphic to \[(\mathbb{C}P^d \# N) \sqcup \mathbb{C}P^{d_1} \sqcup \cdots \sqcup \mathbb{C}P^{d_k}\!\textrm{,}\]
completing the proof.
\end{proof}

\begin{proposition}\label{pro:manifold-realization}
Let $G$ be a compact Lie group and let $F$ be a compact stably parallelizable smooth $(2d+1)$\nobreakdash-manifold for $d\geq 1$, with non-empty boundary. If $F$ is the fixed point set of a finite contractible $G$-CW complex, then there exists a smooth action of $G$ on a complex projective space $\mathbb{C}P^{n+d}$ such that the fixed point set is diffeomorphic to the disjoint union of the connected sum $\mathbb{C}P^d \#\partial F$ and a number of complex projective spaces.
\end{proposition}

\begin{proof}
By Theorem~\ref{thm:isotropy-subgroups}, there exists a complex $G$-module $V$ such that $V^G = 0$ and $\mathscr{F}_\mathrm{iso}(G;X)\subset\mathscr{F}_\mathrm{iso}(G;V)$.
For an integer $\ell \geq 1$, let $\ell\,V = V \oplus \cdots \oplus V$, the direct sum of $\ell$ copies of $V$. Set $n = \frac{1}{2} \ell \dim_{\mathbb{R}} V$. Then $\ell\,V  = \mathbb{R}^{2n}_{\varrho}$, where $\varrho \colon G \to U(n) \subset O(2n)$ is a representation without trivial summand.

By Theorem~\ref{thm:thickening}, if $\ell$ is sufficiently large, then there exists a smooth action of $G$ on the disk $D^{2n+2d+1}$ with fixed point set diffeomorphic to $F$. Moreover, for any point $y \in F$, as real $G$-modules, \[T_y(D^{2n+2d+1}) \cong \mathbb{R}^{2n}_{\varrho} \oplus \mathbb{R}^{2d+1}\!\textrm{,}\] where $G$ acts trivially on $\mathbb{R}^{2d+1}$. The action of $G$ on the disk $D^{2n+2d+1}$ restricts to a smooth action of $G$ on $S^{2n+2d}$ such that the fixed point set is diffeomorphic to $\partial F$. Moreover, as real $G$-modules, \[T_y(S^{2n+2d}) \cong \mathbb{R}^{2n}_{\varrho} \oplus \mathbb{R}^{2d}\!\]for any point $y \in \partial F$. Hence, by Lemma~\ref{lem:manifold-realization}, there exists a smooth action of $G$ on $\mathbb{C}P^{n+d}$ such that the fixed point set is diffeomorphic to the disjoint union of $\mathbb{C}P^d \# \partial F$ and a number of complex projective spaces.
\end{proof}

\bigskip\noindent
{\sl{Proofs of Theorems~\ref{thm:main1}} and\/ \ref{thm:main2}.}
Assume as in Theorem~\ref{thm:main1}, that $G$ is a compact Lie group and $F$ is a compact $\mathbb{Z}$\nobreakdash-acyclic smooth $(2d+1)$\nobreakdash-manifold for $d\geq 1$,
or as in Theorem~\ref{thm:main2}, that $G$ is a compact Lie group such that $G_0$ is non-abelian or $G/G_0$ is not of prime power order, and $F$ is a compact stably parallelizable
smooth $(2d+1)$\nobreakdash-manifold for $d \geq 1$, with $\chi(F) \equiv 1 \pmod {n_G}$.

In both cases, Proposition~\ref{pro:G-CW-complex} asserts that there exists a finite contractible $G$-CW complex $X$ such that $X^G=F$. Therefore, in both cases, it follows 
from Proposition~\ref{pro:manifold-realization} that there exists a smooth action of $G$ on a complex projective space such that the fixed point set is diffeomorphic to 
the disjoint union of $\mathbb{C}P^{d} \#\partial F$ and a number of complex projective spaces. \hfill$\square$

\end{document}